\documentclass[12pt, a4paper, reqno]{amsart}
\usepackage[svgnames]{xcolor}
\usepackage[centering, heightrounded]{geometry}

\usepackage{amsmath, amsthm}
\usepackage{amssymb}
\usepackage{mathrsfs}%
\usepackage{braket}%

\usepackage{secdot}%

\usepackage{enumitem}
\setlist[enumerate]{topsep=2pt, parsep=0pt, partopsep=0pt, itemsep=3pt, leftmargin=3em, label=\rm{(\roman*)}}

\usepackage{cancel}

\usepackage{accents}

\makeatletter
\def\widebar{\accentset{{\cc@style\underline{\mskip8mu}}}}
\makeatother

\numberwithin{equation}{section}

\newtheorem{theorem}{Theorem}[section]
\newtheorem{lemma}[theorem]{Lemma}

\newtheorem{corollary}[theorem]{Corollary}
\newtheorem*{theorem*}{Theorem}
\newtheorem*{lemma*}{Lemma}

\newtheorem*{assumption*}{Assumption}

\theoremstyle{definition}
\newtheorem{definition}[theorem]{Definition}
\newtheorem*{remark*}{Remark}

\theoremstyle{remark}
\newtheorem{remark}[theorem]{Remark}

\let\oldl\l 

\renewcommand{\l}{\left}
\renewcommand{\r}{\right}
\newcommand{\eps}{\varepsilon}

\newcommand{\R}{{\mathbb R}}
\newcommand{\C}{{\mathbb C}}

\newcommand{\T}{{\mathbb T}}

\renewcommand{\Re}{\operatorname{Re}}
\renewcommand{\Im}{\operatorname{Im}}

\newcommand{\pt}{\partial}
\newcommand{\cleq}{\lesssim}

\newcommand{\til}{\widetilde}

\def\tbra[#1,#2]{\left\langle #1 , #2\right\rangle} 
\def\rbra[#1,#2]{\left( #1 , #2 \right)} 

\newcommand{\ce}{\mathrel{\mathop:}=}

\newcommand{\supp}{\operatorname{supp}} 

\def\norm[#1]{\left\Vert #1 \right\Vert}
\def\abs[#1]{\left\vert #1 \right\vert}

\newcommand{\scX}{{\mathscr X}}
\newcommand{\scY}{{\mathscr Y}}
\newcommand{\cD}{{\mathcal D}}

\begin{document}

\title[Uniqueness for the logarithmic Schr\"odinger equation]{Uniqueness of solutions for the logarithmic Schr\"odinger equation}

\author[M. Hayashi]{Masayuki Hayashi}
\address{Graduate School of Human and Environmental Studies,
Kyoto University, Kyoto 606-8501, Japan
\newline\indent
Waseda Research Institute for Science and Engineering, Waseda University, Tokyo 169-8555, Japan
}
 \email{hayashi.masayuki.3m@kyoto-u.ac.jp}

\date{\today}

\begin{abstract}
We consider the Cauchy problem for the logarithmic Schr\"odinger equation and prove uniqueness of weak $H^s(\R^d)$ solutions for $s\in(0,1)$, which improves on the previous uniqueness result in $H^1(\R^d)$. The proof is achieved by combining a nontrivial use of integral equations, local smoothing estimates, and quantitative estimates of the sublinear effect of the nonlinearity, based on the localization argument. We also study uniqueness on the torus and uniqueness of the equation perturbed by pure power nonlinearities.
\end{abstract}

\maketitle



\section{Introduction}

\subsection{Setting of the problem}
\label{sec:1.1}
We consider the logarithmic Schr\"odinger equation
\begin{align}
\label{eq:1.1}
\left\{
\begin{aligned}
&i\pt_t u+\Delta u+\lambda u\log (|u|^2)=0,
\\
& u_{\mid t=0} =u_0,
\end{aligned}
\quad(t,x)\in\R\times\R^d,~\lambda\in\R\setminus\{0\}.
\r.
\end{align}
The charge $\norm[u]_{L^2}^2$ and the energy
\begin{align*}
E (u)=\frac{1}{2}\int|\nabla u|^2 -\frac{\lambda}{2}\int |u|^2\l( \log(|u|^2) -1\r)
\end{align*}
are formally conserved by the flow of \eqref{eq:1.1}, and the equation is formally rewritten by the Hamiltonian form $i\pt_t u=E'(u)$. 
This model was first introduced in \cite{BM76} and later found to be suitable for describing various physical phenomena \cite{BM79, H85, KEB00, Z10}.
It is known that the properties of the solution of \eqref{eq:1.1} rather differ depending on the sign of $\lambda$. When $\lambda>0$, the equation has a non-dispersive structure and is known to have explicit standing waves called Gaussons, and  multi-Gaussons \cite{C83, A16, DMS14, BCST19nm, F21}.
On the other hand, when $\lambda<0$ it was shown in \cite{CG18} that \eqref{eq:1.1} has a universal dispersive structure. The non-dispersive/dispersive structure of \eqref{eq:1.1} is significantly different from that of power-type nonlinear Schr\"odinger equations, and it can be seen that nonlinear effects are strongly evident. More complete references can be found in the recent survey \cite{C22}. 

The main difficulty in the Cauchy problem for \eqref{eq:1.1} is that the nonlinearity breaks the local Lipschitz continuity. The nonlinearity has a sublinear effect due to the singularity of the logarithmic function at the origin, which yields the lack of Lipschitz continuity. To consider solutions to \eqref{eq:1.1} in the Sobolev space $H^s$, it is necessary to consider the differential equation in a local sense.  First, we will give the definition of $H^s$ solutions of \eqref{eq:1.1}.
\begin{definition}
\label{def:1.1}
Let $I\subset\R$ be an interval with $0\in I$ and let $s\ge0$. Given $u_0\in H^s(\R^d)$. 
We say that a function $u$ is a weak $H^s$ solution of \eqref{eq:1.1} on $I$ 
if $u\in L^\infty(I,H^s(\R^d))$ such that (i) $u$ satisfies 
\begin{align}
\label{eq:1.2}
i\pt_t u+\Delta u+\lambda u\log (|u|^2)=0\quad\text{in}~\cD'(I\times \R^d),
\end{align}
i.e., in the sense of distributions on $I\times \R^d$ and (ii) $u_{\mid t=0}=u_0$. 
We say that a function $u$ is a strong $H^s$ solution of \eqref{eq:1.1} on $I$ if $u\in C(I,H^s(\R^d))$ such that (i) and (ii) hold.
\end{definition}
If $u\in L^\infty(I,H^s(\R^d))$ with $s\in[0,1]$ satisfies \eqref{eq:1.2}, then $u$ satisfies 
\begin{align}
\label{eq:1.3}
i\pt_t u+\Delta u +\lambda u\log(|u|^2)=0\quad\text{in}~H^{s-2}(B_R),
\end{align}
for almost all $t\in I$ and any $R>0$, where $B_R$ is an open ball of radius $R$ with center at the origin. We note from this fact that $u_{\mid t=0}=u_0$ makes sense as an element of $H^s(\R^d)$ (see Lemma \ref{lem:2.2}). 
We also note that the nonlinear term does not belong to $H^{s-2}(\R^d)$ in general, so $B_R$ in \eqref{eq:1.3} cannot be replaced by $\R^d$.

Despite the lack of Lipschitz continuity of the nonlinearity, we can expect that the equation \eqref{eq:1.1} has uniqueness properties. Cazenave and Haraux \cite{CH80} introduced the inequality
\begin{align}
\label{CH}
\bigl| \Im\bigl[ (\overline{z-w})( z\log|z|-w\log|w| ) \bigr]\bigr|
\leq |z-w|^2\quad\text{for all}~z,w\in\C,
\end{align}
and constructed a unique strong solution\footnote{Although we do not state it explicitly, all of the solutions constructed in previous works have been global solutions in time.} 
in the energy space 
\begin{align*}
\l\{ f\in H^1(\R^d) : |f|^2\log(|f|^2)\in L^1(\R^d) \r\}
\end{align*}
for $\lambda>0$. Regarding uniqueness, they proved a stronger claim:
\begin{align}
\label{eq:1.5}
\norm[u(t)-v(t)]_{L^2}
\leq
e^{2|\lambda t|} \norm[ u(0)-v(0)]_{L^2}, \quad t\in\R
\end{align}
for any weak $H^1$ solutions $u$ and $v$ to \eqref{eq:1.1}. These results were recently revisited in \cite{HO24}, and strong solutions were constructed in $H^1$ ($\lambda\neq0$), the energy space ($\lambda\neq0$), and the $H^2$ energy space ($\lambda>0$), respectively. We refer the reader to \cite{GLN10, CG18} for the Cauchy problem in weighted Sobolev spaces, which are however narrower than the energy space.

For low regularity solutions, we have the following result:
\begin{theorem}[{\cite{CHO24}}]
\label{thm:1.2}
The flow map $\Phi:u_0\mapsto u$ for \eqref{eq:1.1} is uniquely extended from $H^1$ to $L^2$. If $u_0\in L^2(\R^d)$, $\Phi(u_0)\in C(\R,L^2(\R^d))$ is a strong solution to \eqref{eq:1.1}, and $\Phi$ is Lipschitz continuous:
\begin{align*}
\norm[ \Phi(u_0)(t)-\Phi(v_0)(t)]_{L^2}
\leq
e^{2|\lambda t|} \norm[ u_0- v_0]_{L^2} 
\end{align*}
for any $u_0,v_0\in L^2(\R^d)$ and all $t\in\R$. If in addition $u_0\in H^s(\R^d)$ for $s\in(0,1)$, then $\Phi(u_0)\in C(\R, H^s(\R^d))$.
\end{theorem}
For $H^s$ solutions with $s\in(0,1)$, a priori estimates of $H^s$ were uniformly obtained for approximate equations, but the flow map was extended through $H^1$ solutions.
Therefore, the uniqueness in Theorem \ref{thm:1.2} depends on the method of constructing the solution. Currently known construction methods for low regularity solutions to \eqref{eq:1.1} can be summarized into the following three:
\begin{itemize}
\setlength{\itemsep}{2pt} 
\item Unique extension of the flow map from $H^1$ to $H^s$ for $s\in[0,1)$ \cite{CHO24}

\item Maximal monotone theory \cite{CH80, B73}

\item Compactness methods

\end{itemize}
The first two methods yield uniqueness claims, but as noted above for the first, and similarly for the second, the uniqueness depends on how the solution is constructed. We do not know from previous works whether the solutions constructed by each method coincide. In compactness methods, if we consider different approximation equations, solutions obtained by limiting procedures may be different. In fact, in previous works \cite{CH80, GLN10, CG18, H18, HO24}, the methods of approximating nonlinearities were slightly different. Overall, uniqueness of low regularity solutions was left unclear.
\subsection{Main results}

In this paper we prove strong uniqueness properties to logarithmic Schr\"odinger equations. 
Our first main result is the following:
\begin{theorem}
\label{thm:1.3}
Let $T>0$ and let $s\in(0,1)$. For any $u_0\in H^s(\R^d)$, weak $H^s$ solutions of \eqref{eq:1.1} on $[0,T]$ are unique.
\end{theorem}
This theorem allows us to strengthen the claim on $H^s$ solutions in Theorem \ref{thm:1.2}.
\begin{corollary}
\label{cor:1.4}
Let $s\in(0,1)$. For any $u_0\in H^s(\R^d)$, there exists a unique strong solution $u\in C(\R, H^s(\R^d))$ to \eqref{eq:1.1}.
\end{corollary}
The main theorem corresponds to an \emph{unconditional uniqueness} claim.
This notion was introduced by Kato \cite{K95}, where uniqueness of a power-type nonlinear Schr\"odinger equation was studied based on Strichartz's estimates. In the equation \eqref{eq:1.1}, a simple application of Strichartz's estimates does not work well due to the sublinear effect of the nonlinearity. That may be why the only known way to show unconditional uniqueness was through a simple combination of the inequality \eqref{CH} and the differential equation, which requires $H^1$ solutions (see Section \ref{sec:2.1}).

For the proof of Theorem \ref{thm:1.3}, we take advantage of integral equations based on the localization argument. By using integral equations and inner products to expand the $L^2$ norm of the difference between two solutions, we can apply \eqref{CH} to nonlinear estimates. This use of integral equations is inspired by \cite{O06}. It would be of independent interest that \eqref{CH} still works even in the framework of integral equations. 
Another ingredient in the proof is that we use local smoothing estimates in \cite{KPV93} to handle errors coming from the linear term. In combination with this smoothing, we use the following quantitative estimate on the nonlinearity:
\begin{align*}
\abs[ u\log(|u|^2)]\cleq \frac{1}{\delta}|u|^{1-\delta}+|u|\log^+|u|,\quad \delta\in(0,1),
\end{align*}
where the explicit dependence $1/\delta$ on the right-hand side is the key to our argument. Then, we introduce the relation
\begin{align}
\label{eq:1.6}
\delta=\frac{1}{\log R}
\end{align}
for $\delta$ and the cutoff parameter\footnote{We use the cut-off function $\varphi_R$ with $\supp\varphi_R\subset B_R$ in a localization argument.} $R$,
which enables us to eliminate the errors in the limit $R\to\infty$. 

Next, we present the uniqueness result on the torus. In view of the framework considered in numerical simulations 
(see, e.g., \cite{BCST19siam, BCST19nm}), it would be important to consider the Cauchy problem for
\begin{align}
\label{eq:1.7}
\left\{
\begin{aligned}
&i\pt_t u+\Delta u+\lambda u\log (|u|^2)=0,
\\
& u_{\mid t=0} =u_0,
\end{aligned}
\quad(t,x)\in\R\times\T^d,~\lambda\in\R\setminus\{0\}.
\r.
\end{align}
Weak solutions of \eqref{eq:1.7} can be defined in the same way as in Definition \ref{def:1.1}, and the same result of Theorem \ref{thm:1.2} holds true on the torus as well. Our uniqueness result is as follows.
\begin{theorem}
\label{thm:1.5}
Let $T>0$ and let $s\in(0,1)$. For any $u_0\in H^s(\T^d)$, weak $H^s$ solutions of \eqref{eq:1.7} on $[0,T]$ are unique. When $d=1$, for any $u_0\in L^2(\T)$, weak $L^2$ solutions of \eqref{eq:1.7} on $[0,T]$ are unique.
\end{theorem}
Since the torus does not require the localization argument, the proof for $s>0$ is easier than Theorem \ref{thm:1.3}. 
Uniqueness of $L^2$ solutions is still a delicate problem even on the torus, but we can prove it when $d=1$. To control the logarithmic growth of the nonlinearity, we apply the classical estimate by Zygmund \cite{Z74}:  
\begin{align}
\label{eq:1.8}
\norm[e^{it\pt_x^2}u_0]_{L^4([0,T]\times\T)}\cleq \norm[u_0]_{L^2(\T)}.
\end{align}
When using this estimate, we can no longer use \eqref{CH}, but instead we use the inequality
\begin{align}
\label{eq:1.9}
\bigl|z\log|z|-w\log|w| \bigr|&\cleq \frac{1}{\delta} |z-w|^{1-\delta}\quad\text{for all}~ |z|,|w|\leq 1,~\delta\in(0,1),
\end{align}
to estimate the sublinear effect of the nonlinearity.
Combining this with \eqref{eq:1.8}, we can reduce the proof to a Gronwall type lemma using $\delta$ as a parameter tending to zero.

Finally, we consider the equation perturbed by pure power nonlinearities: 
\begin{align}
\label{eq:1.10}
\l\{
\begin{aligned}
&i\pt_t u +\Delta u+ \lambda u\log(|u|^2)+\mu|u|^{\alpha}u=0,
\\
&u_{\mid t=0}=u_0,
\end{aligned}
\r.
\quad(t,x)\in\R\times\R^d,~\lambda, \mu\in\R\setminus\{0\},
\end{align}
with $0<\alpha<\frac{4}{(d-2)_+}$.
In \cite{CG18} the authors proved that when $\lambda<0$, $\mu<0$, this equation has the same universal dispersive structure as \eqref{eq:1.1} with $\lambda<0$. They constructed the solution of \eqref{eq:1.10} in the weighted Sobolev space $\{f\in H^1(\R^d) : |x|f \in L^2(\R^d) \}$, 
but uniqueness of solutions was shown only for the case $d=1$. The following result covers all cases that were left unproven regarding uniqueness, and contributes to strengthening their result.
%
\begin{theorem}
\label{thm:1.6}
Let $T>0$. For any $u_0\in H^1(\R^d)$, weak $H^1$ solutions of \eqref{eq:1.10} on $[0,T]$ are unique.
\end{theorem}
For the proof of Theorem \ref{thm:1.6}, we estimate the difference of two solutions by combining Strichartz's estimates and \eqref{eq:1.9}, based on the localization argument. Similarly to Theorem \ref{thm:1.3}, by setting the relation \eqref{eq:1.6} for $\delta$ in \eqref{eq:1.9} and the cutoff parameter $R$, we can eliminate the errors in the limit $R\to\infty$.
\\

To the best of the author's knowledge, our arguments in this paper are the first to effectively apply dispersive/smoothing estimates of the Schr\"odinger group to logarithmic Schr\"odinger equations. We expect that the strategy is robust and can be applied to other types of logarithmic evolution equations. We close this introduction by emphasizing that we have provided new results on the uniqueness problem for logarithmic Schr\"odinger equations, the topic on which no essential progress has been made for a long time since \cite{CH80}.

\subsection{Organization of the paper}

The rest of the paper is organized as follows. In Section \ref{sec:2.1} we review the previous uniqueness result of weak $H^1$ solutions. Then, we organize the localization argument for $H^s$ solutions in Section \ref{sec:2.2}, and prove Theorem \ref{thm:1.3} in Section \ref{sec:2.3}. 
In Section \ref{sec:3.1} we prepare algebraic inequalities for logarithmic nonlinearities and recall smoothing estimates of the Schr\"odinger group. In Section \ref{sec:3.2} we study uniqueness for \eqref{eq:1.7} and prove Theorem \ref{thm:1.5}. In Section \ref{sec:3.3} we study uniqueness for \eqref{eq:1.10} and prove Theorem \ref{thm:1.6}.

\subsection*{Notation}
We often use the abbreviated notation such as
\begin{align*}
L^\infty_TX=L^\infty([0,T],X),\quad T>0
\end{align*}
for a Banach space $X$. 
We write the Duhamel term by
\begin{align*}
\Phi[f](t) =\int_0^t U(t-\tau)f(\tau)d\tau,\quad U(t)=e^{it\Delta}
\end{align*}
for a time-dependent function $f(t)$.
The Fourier transform on the whole space is defined by
\begin{align*}
\hat{f}(\xi) =\int_{\R^d} f(x) e^{-2\pi i\xi\cdot x}dx,\quad \xi\in\R^d.
\end{align*}
We define the fractional derivative by
\begin{align*}
(D^sf)(x)= \int_{\R^d}(2\pi |\xi|)^s\hat{f}(\xi)e^{2\pi i \xi\cdot x} d\xi,\quad x\in\R^d,~s\in\R.
\end{align*}

We use $A\cleq B$ to denote the inequality $A\le CB$ for some constant $C>0$. The dependence of $C$ 
is usually clear from the context and we often omit this dependence. 
We sometimes use $A\cleq_* B$ to clarify the dependence of a constant.


\section{Uniqueness in $H^s(\R^d)$}

In this section we study uniqueness of weak $H^s$ solutions to \eqref{eq:1.1}.

\subsection{Review on the previous result}
\label{sec:2.1}

We first review the uniqueness result in \cite{CH80}. 
Let $u,v$ be weak $H^1$ solutions of \eqref{eq:1.1} on $[0,T]$. 
Take a function $\varphi\in C^{\infty}_c(\R^d)$ satisfying
\begin{align*}
\varphi (x)=
\left\{
\begin{aligned}
&1&&\text{if}~|x|\leq 1/2,
\\
&0 &&\text{if}~|x|\geq 1,
\end{aligned}
\r.
\qquad 0\leq\varphi (x)\leq 1\quad\text{for all}~x\in\R^d.
\end{align*}
We set the cut-off function 
\begin{align}
\label{eq:2.1}
\varphi_R(x)=\varphi(x/R),\quad x\in\R^d,~R>0.
\end{align}
We note that weak $H^1$-solutions $u$ and $v$ satisfy \eqref{eq:1.3} with $s=1$. Using this fact and the inequality \eqref{CH}, we get
\begin{align*}
\frac{d}{dt}\norm[\varphi_R(u-v)]_{L^2}^2
&=2\Im\tbra[i\pt_t(u-v),\varphi_R^2(u-v)]_{H^{-1}(B_R), H^1_0(B_R)}
\\
&=
\begin{aligned}[t]
&2\Im\rbra[\nabla(u-v),\nabla(\varphi_R^2)(u-v) ]_{L^2}
\\
&{}-4\lambda\Im\rbra[u\log |u|-v\log|v| ,
\varphi_R^2(u-v)]_{L^2}
\end{aligned}
\\
&\leq 
\frac{C(M)}{R}+4\abs[\lambda]\norm[\varphi_R(u-v)]_{L^2}^2.
\end{align*}
Applying Gronwall's lemma, we obtain
\begin{align*}
\norm[\varphi_R(u-v)(t)]_{L^2}^2\leq e^{4|\lambda| t} \l(\norm[ (u-v)(0)]_{L^2}^2+\frac{C(M)}{R}T \r)
\end{align*}
for all $t\in[0,T]$. Therefore, by applying Fatou's lemma, we obtain
\begin{align}
\label{eq:2.2}
\norm[(u-v)(t)]_{L^2}^2
\leq\liminf_{R\to\infty}\norm[\varphi_R(u-v)(t)]_{L^2}^2\leq
e^{4|\lambda| t} \norm[ (u-v)(0)]_{L^2}^2,
\end{align}
which yields the $L^2$ Lipschitz flow on $[0,T]$. 
In particular, this implies that 
if  $u(0)=v(0)$, then $u=v$ on $[0,T]$. 
\begin{remark}
\label{rem:2.1}
The final estimate \eqref{eq:2.2} is meaningful for $L^2$ solutions, but we need to use $H^1$ solutions for this derivation. As mentioned in Section \ref{sec:1.1}, we need to consider the differential equation in a local sense, and it can be seen that $H^1$ estimates are used to control the error that arises from the localization argument.
\end{remark}

\subsection{Localization argument}
\label{sec:2.2}

We now study uniqueness of weak $H^s$ solutions of \eqref{eq:1.1} for $s\in(0,1)$.
We begin with the following lemma.
\begin{lemma}
\label{lem:2.2}
Let $s\in[0,1]$ and let $T>0$. Assume that $u\in L^\infty([0,T], H^s(\R^d))$ 
satisfies \eqref{eq:1.2} with $I=[0,T]$.
Then, $u\in C_w([0,T], H^s(\R^d))$ and in particular $u(0)$ has the meaning as an element of $H^s$.
\end{lemma}
\begin{proof}
From the assumption one can easily prove that for any $R>0$,
\begin{align}
\label{eq:2.3}
u\in W^{1,\infty}([0,T], H^{s-2}(B_R))
\end{align} 
and 
\begin{align}
\label{eq:2.4}
i\pt_t u+\Delta u +\lambda u\log(|u|^2)=0\quad\text{in}~H^{s-2}(B_R)
\end{align}
for almost all (a.a.)~${t\in[0,T]}$.
It follows from \eqref{eq:2.3} that $u\in C([0,T], H^{s-2}(B_R))$, and the interpolation implies that $u\in C([0,T], H^{s-\eps}(B_R))$ for any small $\eps>0$. 
Thus, for any $\psi\in C^\infty_c(\R^d)$ it follows from the continuity property of $u$ in time that the function
\begin{align*}
t\mapsto \int_{\R^d} u(t)\psi dx
\end{align*}
is continuous. Therefore, from $u\in L^\infty([0,T],H^s(\R^d))$ and a density argument, we conclude that $u\in C_w([0,T], H^s(\R^d))$. 
\end{proof}
Let $u\in L^\infty([0,T], H^s(\R^d))$ be a weak $H^s$ solution of \eqref{eq:1.1}. Since $u$ satisfies \eqref{eq:2.4}, we deduce that
\begin{align*}
i\pt_t (\varphi_R u)+\Delta(\varphi_R u)-2\nabla u\cdot\nabla\varphi_R-u\Delta\varphi_R+\lambda\varphi_Ru\log(|u|^2)=0\quad\text{in}~H^{s-2}(\R^d)
\end{align*}
for a.a. $t\in[0,T]$. By Duhamel's formula, we have
\begin{align}
\label{eq:2.5}
\varphi_Ru(t)=
\begin{aligned}[t]
U(t)\varphi_Ru_0&-i\int_0^tU(t-\tau)\l(2\nabla\varphi_R\cdot\nabla u(\tau) +\Delta\varphi_R u(\tau)\r)d\tau
\\
&{}~+i\int_0^t U(t-\tau)\varphi_Rg(u(\tau))d\tau, \quad t\in [0,T],
\end{aligned}
\end{align}
where we set 
\begin{align}
\label{eq:2.6}
g(u)=\lambda u\log(|u|^2).
\end{align}
To estimate the second term on the right-hand side (RHS) of \eqref{eq:2.5}, we use the following local smoothing estimate.
\begin{lemma}[Local smoothing estimates]
\label{lem:2.3}
Let $s\in(0,1]$. Then, we have
\begin{align*}
\norm[ {D^s\Phi[\chi_{B_R}f]}]_{L^2([0,T]\times B_R)}\cleq R^s\norm[\chi_{B_R}f]_{L^2([0,T]\times\R^d)}
\end{align*}
for all $f\in L^2([0,T]\times\R^d)$.
\end{lemma}
\begin{proof}
When $s=1$, this estimate corresponds to the local smoothing effect for the inhomogeneous case in \cite{KPV93}. Since the estimate is trivial for $s=0$, by Stein's interpolation theorem \cite{S56}, the result follows for $s\in(0,1)$.
\end{proof}
\begin{lemma}
\label{lem:2.4}
Let $s\in(0,1)$ and let $u\in L^\infty([0,T],H^s(\R^d))$. Then, 
\begin{align*}
\norm[ {\Phi[\nabla\varphi_R\cdot\nabla u]} ]_{L^2([0,T]\times B_R)}\cleq
\l(R^{-s}+o(R^{-1}) \r)T^{1/2}\norm[u]_{L^\infty_TH^s(\R^d)}
\end{align*}
as $R\to\infty$. 
\end{lemma}
\begin{proof}
We rewrite the term $\nabla\varphi_R\cdot\nabla u$ as 
\begin{align}
\nabla\varphi_R\cdot\nabla u
&=\nabla\varphi_R\cdot D^{1-s}D^sD^{-1}\nabla u
\notag\\
&=D^{1-s} (\nabla\varphi_R\cdot D^s\til{u} )-\bigl[ D^{1-s}\l(\nabla\varphi_R\cdot D^s\til{u}\r)-\nabla\varphi_R \cdot D^{1-s}D^s\til{u}\bigr],
\label{eq:2.7}
\end{align}
where we set 
\begin{align*}
\til{u}(x)=(D^{-1}\nabla u)(x) =
\int_{\R^d}\frac{i\xi}{|\xi|}\hat{u}(\xi) e^{2\pi i\xi\cdot x}d\xi.
\end{align*}
Applying the first term on the RHS of \eqref{eq:2.7} to Lemma \ref{lem:2.3}, we obtain
\begin{align*}
\norm[ {\Phi [D^{1-s}(\nabla\varphi_R\cdot D^s\til{u}) ] } ]_{L^2([0,T]\times B_R)}
&\cleq R^{1-s}\norm[\nabla\varphi_R\cdot D^s\til{u}]_{L^2([0,T]\times\R^d)}
\\
&\cleq R^{-s}\norm[D^su]_{L^2_TL^2(\R^d)}.
\end{align*}
For the second term on the RHS of \eqref{eq:2.7}, we use 
the fractional Leibniz rule
\begin{align*}
\norm[D^\alpha(fg)-fD^\alpha g]_{L^2(\R^d)}\cleq \norm[D^\alpha f]_{L^\infty(\R^d)}\norm[g]_{L^2(\R^d)},\quad\alpha\in(0,1),
\end{align*}
see \cite{KPV93cpam, L19}.
Then, we deduce that
\begin{align*}
&\norm[
   {\Phi[ D^{1-s} (\nabla\varphi_R\cdot D^s\til{u})-\nabla\varphi_R\cdot D^{1-s}D^s\til{u}] }  ]_{L^2([0,T]\times B_R)} 
\\
\cleq&{}~T^{1/2}\norm[D^{1-s} (\nabla\varphi_R\cdot D^s\til{u})-\nabla\varphi_R\cdot D^{1-s}D^s\til{u} ]_{L^\infty_TL^2(\R^d)}
\\
\cleq&{}~T^{1/2}\norm[D^{1-s}\nabla\varphi_R]_{L^\infty(\R^d)}\norm[D^su]_{L^\infty_TL^2(\R^d)}\cleq T^{1/2}R^{-2+s}\norm[D^su]_{L^\infty_TL^2(\R^d)}.
\end{align*}
Gathering these estimates, we get the result.
\end{proof}

\subsection{Proof of unconditional uniqueness}
\label{sec:2.3}
We begin with the following simple lemma, which is the basis for the argument using integral equations in \cite{O06}.
\begin{lemma}
\label{lem:2.5}
Let $f\in L^2([0,T], L^2)$. Then,
\begin{align*}
\norm[\int_0^tf(\tau)d\tau]_{L^2}^2
=2\Re\int_0^t\rbra[f(\tau),\int_0^{\tau}f(\tau^{\prime})d\tau^{\prime}]_{L^2}d\tau
\end{align*}
for a.a. $t\in[0,T]$.
\end{lemma}
\begin{proof}
We provide a proof for completeness. We first note that
\begin{align*}
\norm[\int_0^tf(\tau)d\tau]_{L^2}^2=\iint_{[0,t]\times[0,t]}\rbra[f(\tau),f(\tau^{\prime})]_{L^2}d\tau d\tau^{\prime}.
\end{align*}
Next, dividing the rectangular $[0,t]\times[0,t]$ into two triangles and using the Fubini-Tonelli Theorem, the RHS can be rewritten as
\begin{align*}
&\iint_{0\le\tau\le\tau^{\prime}\le t}\rbra[f(\tau),f(\tau^{\prime})]_{L^2}d\tau d\tau^{\prime}
+\iint_{0\le\tau^{\prime}\le\tau\le t}\rbra[f(\tau),f(\tau^{\prime})]_{L^2}d\tau d\tau^{\prime}
\\
=&\int_0^t\rbra[\int_0^{\tau^{\prime}}f(\tau)d\tau,f(\tau^{\prime})]_{L^2}d\tau^{\prime}
+\int_0^t\rbra[f(\tau),\int_0^{\tau}f(\tau^{\prime})d\tau^{\prime}]_{L^2}d\tau
\\
=&2\Re\int_0^t\rbra[f(\tau),\int_0^{\tau}f(\tau^{\prime})d\tau^{\prime}]_{L^2}d\tau,
\end{align*}
which completes the proof.
\end{proof}
The following result follows from Lemma \ref{lem:3.1}, which we will prove in Section \ref{sec:3}.
\begin{lemma}
\label{lem:2.6}
There exist $c_1,c_2>0$ such that for any $\delta\in(0,1)$ and any $z\in\C$,
\begin{align*}
\abs[g(z)]\leq \frac{c_1}{\delta}|z|^{1-\delta}+c_2|z|\log^+|z|.
\end{align*}
\end{lemma}
We are now in a position to prove Theorem \ref{thm:1.3}.
\begin{proof}[Proof of Theorem \ref{thm:1.3}]
Given initial data $u_0\in H^s(\R^d)$ for $s\in(0,1)$. Let $u,v$
be two weak $H^s$ solutions of \eqref{eq:1.1} on $[0,T]$.
Our goal is to prove that $u=v$ on $[0,T]$. 

We set 
\begin{align*}
M\ce \max\left\{ \norm[u]_{L^\infty_TH^s},  \norm[v]_{L^\infty_TH^s}\r\}.
\end{align*}
From the integral equation \eqref{eq:2.5}, we have
\begin{align}
\label{eq:2.8}
\varphi_R(u-v)(t)
=e_R(t)+i\int_0^tU(t-\tau)\varphi_R (g(u)-g(v) )(\tau) d\tau,\quad t\in[0,T],
\end{align}
where $e_R(t)$ is defined by
\begin{align}
\label{eq:2.9}
e_R(t)={-}i\int_0^tU(t-\tau)\bigl(2\nabla\varphi_R\cdot\nabla(u-v)(\tau)+\Delta\varphi_R (u-v)(\tau)\bigr)d\tau.
\end{align}
We note that by Lemma \ref{lem:2.4},
\begin{align}
\label{eq:2.10}
\norm[e_R]_{L^2_TL^2(B_R)}\cleq_M \frac{T^{1/2}}{R^s}+o\l(\frac{1}{R}\r).
\end{align}
 
For the nonlinearity, it follows from Lemma \ref{lem:2.6} that 
\begin{align*}
\abs[g(u)]\cleq \frac{1}{\delta}|u|^{1-\delta}+|u|^{1+\delta_0},
\end{align*}
where $\delta\in(0,1)$ is treated as a parameter and $\delta_0$ is fixed as satisfying $\delta_0\in(0,\frac{2s}{(d-2s)}_+)$. By H\"older's inequality and Sobolev's embedding, we obtain
\begin{align}
\label{eq:2.11}
\norm[\varphi_Rg(u)]_{L^2}
\cleq
\frac{1}{\delta}|B_R|^{\delta/2}\norm[u]_{L^2}^{1-\delta}+\norm[u]_{H^s}^{1+\delta_0}.
\end{align}
Applying Lemma \ref{lem:2.5}, we deduce that
\begin{align*}
&\norm[\int_0^tU(t-\tau)\varphi_R(g(u)-g(v)) (\tau)d\tau]_{L^2}^2
\\
={}&2\Im\int_0^t
\rbra[\varphi_R(g(u)-g(v))(\tau), i\int_0^\tau U(\tau-\tau')\varphi_R(g(u)-g(v))(\tau')d\tau']_{L^2}d\tau
\\
={}&
\begin{aligned}[t]
&2\Im\int_0^t\Bigl(\varphi_R(g(u)-g(v) )(\tau), \varphi_R(u-v)(\tau)\Bigr)_{L^2}d\tau
\\
&{}-2\Im\int_0^t\Bigl(\varphi_R(g(u)-g(v))(\tau), e_R(\tau)\Bigr)_{L^2}d\tau,
\end{aligned}
\end{align*}
where we used \eqref{eq:2.8} in the last equality.
Applying the inequality \eqref{CH}, we obtain
\begin{align*}
\Im\int_0^t\Bigl(\varphi_R(g(u)-g(v))(\tau), \varphi_R(u-v)(\tau)\Bigr)_{L^2}d\tau
\le 2\abs[\lambda]\int_0^t\norm[\varphi_R(u-v)(\tau)]_{L^2}^2d\tau.
\end{align*}
By \eqref{eq:2.10} and \eqref{eq:2.11}, we obtain
\begin{align*}
\abs[\int_0^t\Bigl(\varphi_R(g(u)-g(v))(\tau), e_R(\tau)\Bigr)_{L^2}d\tau]
&\cleq_M \frac{T^{1/2}|B_R|^{\delta/2}}{\delta}\norm[e_R]_{L^2_TL^2(B_R)}
\cleq_M \frac{|B_R|^{\delta/2}}{\delta R^s}T.
\end{align*}
Gathering these estimates, we deduce that
\begin{align}
\label{eq:2.12}
\norm[\int_0^tU(t-\tau)\varphi_R(g(u)-g(v))(\tau)d\tau]_{L^2}^2
\cleq&\int_0^t\norm[\varphi_R(u-v)(\tau)]_{L^2}^2d\tau+\frac{|B_R|^{\delta/2}}{\delta R^s}T.
\end{align}

So far, we have two independent parameters $\delta$ and $R$, but here we set {$\delta=\frac{1}{\log R}$}.
We note that this yields that $|B_R|^{\delta/2}\cleq1$. Here we set
\begin{align*}
y_R(t)=\norm[\varphi_R(u-v)(t)]_{L^2}^2,\quad z_R(t)=\int_0^ty_R(\tau)d\tau\quad\text{for}~ t\in[0,T],~R>0.
\end{align*}
Taking the $L^2(B_R)$ norm on both sides of \eqref{eq:2.8}, we obtain from \eqref{eq:2.10} and \eqref{eq:2.12} that
\begin{align*}
y_R(t)&=
\norm[\varphi_R(u-v)(t)]_{L^2(B_R)}^2
\\
&\cleq \norm[e_R(t)]_{L^2(B_R)}^2+\frac{\log R}{R^s}T+z_R(t).
\end{align*}
By integrating both sides over $[0,t]$ and using \eqref{eq:2.10}, we get
\begin{align*}
z_R(t)
&\cleq \frac{\log R}{R^s}T^2+\int_0^tz_R(\tau)d\tau
\end{align*}
for large $R>0$ and $t\in[0,T]$. Applying Gronwall's lemma, we obtain
\begin{align*}
z_R(T)=\int_0^T\norm[\varphi_R(u-v)(\tau)]_{L^2}^2
\leq \frac{\log R}{R^s}CT^2\exp(CT),
\end{align*}
where $C>0$ is a constant independent of $R$.
Therefore, passing to the limit $R\to\infty$, we conclude that $u=v$ on $[0,T]$.
\end{proof}
\begin{remark}
\label{rem:2.7}
When $s=0$, it follows from Lemma \ref{lem:2.3} that 
\begin{align*}
\norm[e_R]_{L^2([0,T]\times B_R)}=o(1)\quad\text{as}~R\to\infty,
\end{align*}
however which is not enough to conclude uniqueness in $L^\infty([0,T],L^2(\R^d))$ in our proof.
Unconditional uniqueness of $L^2$ solutions to \eqref{eq:1.1} has been an open problem since \cite{CH80} (see \cite[Remark 9.3.7]{C03}).
\end{remark}

\section{Further results}
\label{sec:3}

In this section we study uniqueness for \eqref{eq:1.7} and uniqueness for \eqref{eq:1.10}.
\subsection{Preliminaries}
\label{sec:3.1}
We take a function $\theta\in C^{1}_c(\C,\R)$ satisfying
\begin{align*}
\theta (z)=
\left\{
\begin{aligned}
&1&&\text{if}~|z|\leq 1/2,
\\
&0 &&\text{if}~|z|\geq 1,
\end{aligned}
\r.
\qquad 0\leq\theta (z)\leq 1\quad\text{for all}~z\in\C.
\end{align*}
We decompose the nonlinearity \eqref{eq:2.6} as
\begin{align*}
g_1(u)=\theta(u)g(u), \quad g_2(u)=(1-\theta(u))g(u).
\end{align*}
Then, we have the following.
\begin{lemma}
\label{lem:3.1}
There exist $c_1,c_2>0$ such that for any $\delta\in(0,1)$ and any $z,w\in\C$,
\begin{align}
\label{eq:3.1}
\abs[g_1(z)-g_1(w)]&\leq \frac{c_1}{\delta} |z-w|^{1-\delta},
\\
\label{eq:3.2}
\abs[g_2(z)-g_2(w)]&\leq c_2\l(1+\log^+|z|+\log^+|w|\r)|z-w|.
\end{align}
\end{lemma}
\begin{proof}
Since \eqref{eq:3.2} follows immediately, we only prove \eqref{eq:3.1}. To prove this, it is sufficient to show the inequality
\begin{align}
\label{eq:3.3}
\abs[g(z)-g(w)]&\leq \frac{c}{\delta} |z-w|^{1-\delta}\quad\text{for}~ |z|,|w|\leq 1,
\end{align}
where $c>0$ is some constant independent of $\delta\in(0,1)$ and $z,w\in\C$ with $|z|,|w|\leq 1$.
We may assume $|z|\leq |w|$. We note that
\begin{align*}
\abs[z\log|z|-w\log|w|]
&=\abs[ z\left(\log|z|-\log|w|\r)+(z-w)\log|w| ]
\\
&\leq |z-w|+|z-w|^{1-\delta}|z-w|^{\delta}\abs[\log|w|].
\end{align*}
If we set $|w|=e^{-s}$ for some $s\in[0,\infty)$, then we have
\begin{align*}
|z-w|^{\delta}\abs[\log|w|] &\cleq |w|^\delta \abs[\log|w|]
=\frac{s}{e^{s\delta}}
\leq \frac{s}{1+s\delta}\leq \frac{1}{\delta}.
\end{align*}%
Therefore, we get
\begin{align*}
\abs[z\log|z|-w\log|w|] \cleq |z-w|+\frac{1}{\delta}|z-w|^{1-\delta}\cleq \frac{1}{\delta}|z-w|^{1-\delta},
\end{align*}
which implies \eqref{eq:3.3}.
\end{proof}
The inequality \eqref{eq:3.1} with some constant $c(\delta)>0$ instead of $c_1/\delta$ has been used in \cite{AC17, HO24, CHO24}, but no attention was paid to the explicit dependence of $c(\delta)$ in those studies. The key to our proof is that the constant in \eqref{eq:3.1} can be taken as $1/\delta$.

For smoothing estimates on the one-dimensional torus, the following estimate is proved by Zygmund \cite{Z74}:
\begin{align*}
\norm[e^{it\pt_x^2}u_0]_{L^4([0,T]\times\T)}\cleq \norm[u_0]_{L^2(\T)}.
\end{align*}
Combined with a duality argument \cite{GV92}, this yields the following result.
\begin{lemma}
\label{lem:3.2}
Let $(q,r)$ and $(\gamma,\rho)$ be $(\infty,2)$ or $(4,4)$. Let $I$ be a bounded interval. If $f\in L^{\gamma'}(I, L^{\rho'}(\T))$, then
\begin{align*}
\norm[ {\Phi[f]}]_{L^q(I,L^r(\T))}\cleq \norm[f]_{L^{\gamma'}(I, L^{\rho'}(\T))}.
\end{align*}
\end{lemma}

For the proof of Theorem \ref{thm:1.6}, we use Strichartz's estimates (see, e.g., \cite{S77, Y87, C03}). 
We say that a pair $(q,r)$ is admissible if 
\begin{align*}
\frac{2}{q}=d\l(\frac{1}{2} -\frac{1}{r}\r),\quad 2\le r\le \frac{2d}{(d-2)_+},
\end{align*}
with the exception $(d,q,r)=(2,2,\infty)$. 
\begin{lemma}
\label{lem:3.3}
Let $(q,r)$ and $(\gamma,\rho)$ be admissible and let $I$ be an interval. If $f\in L^{\gamma'}(I, L^{\rho'}(\R^d))$, then
\begin{align*}
\norm[ {\Phi[f]}]_{L^q(I,L^r(\R^d))}\cleq \norm[f]_{L^{\gamma'}(I, L^{\rho'}(\R^d))}.
\end{align*}
\end{lemma}
Finally, we prepare the following Gronwall type lemma.
\begin{lemma}
\label{lem:3.4}
Let $T>0$, $a,b\ge0$, $\delta\in(0,1)$, and let $f\in L^1([0,T])$ be a nonnegative function. If 
\begin{align}
\label{eq:3.4}
f(t)\le a +\frac{b}{\delta}\int_0^tf(\tau)^{1-\delta}d\tau,\quad\text{a.a.}~ t\in[0,T],
\end{align}
then we have
\begin{align}
\label{eq:3.5}
f(t)\leq (a^\delta + bt)^{1/\delta},\quad\text{a.a.}~ t\in[0,T].
\end{align}
\end{lemma}
\begin{proof}
Let the function on the RHS of \eqref{eq:3.4} be $F(t)$. It follows from \eqref{eq:3.4} that
\begin{align*}
(F(t)^\delta)'=\delta F(t)^{\delta-1}\cdot\frac{b}{\delta}f(t)^{1-\delta}\leq b
\end{align*}
for a.a. $t\in[0,T]$, and this implies \eqref{eq:3.5} after integration.
\end{proof}

\subsection{Uniqueness on the torus}
\label{sec:3.2}

In this subsection, we study uniqueness for \eqref{eq:1.7}. If $u\in L^\infty([0,T], H^s(\T^d))$ is a weak $H^s$ solution of \eqref{eq:1.7}, then 
\begin{align*}
i\pt_t u +\Delta u +g(u)=0 \quad \text{in}~H^{s-2}(\T^d)
\end{align*}
for a.a. $t\in[0,T]$. By Duhamel's formula, $u$ satisfies the integral equation
\begin{align}
\label{eq:3.6}
u(t)
=U(t)u_0+i\int_0^tU(t-\tau)g(u(\tau)) d\tau,\quad t\in[0,T].
\end{align}
We are now in a position to prove Theorem \ref{thm:1.5}.
\begin{proof}[Proof of Theorem \ref{thm:1.5}]
For a given initial data $u_0\in H^s(\T^d)$, let $u,v$
be two weak $H^s$ solutions of \eqref{eq:1.7} on $[0,T]$.
We set the constant 
\begin{align*}
M= \max\left\{ \norm[u]_{L^\infty_TH^s},  \norm[v]_{L^\infty_TH^s}\r\}.
\end{align*}
It follows from \eqref{eq:3.6} that ${u-v}$ satisfies the integral equation
\begin{align}
\label{eq:3.7}
(u-v)(t)
=i\int_0^tU(t-\tau)(g(u)-g(v) )(\tau) d\tau,\quad t\in[0,T].
\end{align}

When $s>0$, uniqueness can be easily shown by following a similar procedure to the proof of Theorem 1. We first note that $g(u)\in L^\infty_TL^2$ by Lemma \ref{lem:2.6} and Sobolev's embedding. Then, applying Lemma \ref{lem:2.5} and \eqref{CH}, we deduce that
\begin{align*}
\norm[(u-v)(t)]_{L^2}^2
&=\norm[\int_0^tU(t-\tau)(g(u)-g(v)) (\tau)d\tau]_{L^2}^2
\\
&=2\Im\int_0^t\Bigl( (g(u)-g(v) )(\tau), (u-v)(\tau)\Bigr)_{L^2}d\tau
\\
&\leq 4\abs[\lambda]\int_0^t\norm[(u-v)(\tau)]_{L^2}^2d\tau
\end{align*}
for all $t\in[0,T]$. Applying Gronwall's lemma, we conclude that $u=v$ on $[0,T]$.

We now consider the case of $s=0$ and $d=1$. We introduce the function space
\begin{align*}
\scX(I)=L^\infty(I, L^2(\T))\cap L^4(I, L^4(\T)),
\end{align*}
where $I$ is a subinterval of $[0,T]$, and its norm 
\begin{align*}
\norm[u]_{\scX(I)}=\norm[u]_{L^\infty(I,L^2)}+\norm[u]_{L^4(I,L^4)}.
\end{align*}
It follows from \eqref{eq:3.1} and Lemma \ref{lem:3.2} that
\begin{align}
\label{eq:3.8}
\norm[ {\Phi[g_1(u)-g_1(v)]} ]_{\scX([0,t])}
\cleq \frac{1}{\delta}\int_0^t\norm[(u-v)(\tau)]_{L^2}^{1-\delta}d\tau
\end{align}
for $\delta\in(0,1)$ and $t\in[0,T]$. By \eqref{eq:3.2}, we have
\begin{align*}
\abs[g_2(u)-g_2(v)]&\cleq\l(|u|+|v|\r)|u-v|.
\end{align*}
Then, applying Lemma \ref{lem:3.2} and H\"older's inequality, we obtain
\begin{align*}
\norm[ {\Phi[g_2(u)-g_2(v)] } ]_{\scX([0,t])}
&\cleq
(\norm[u]_{L^\infty_t L^2}+\norm[u]_{L^\infty_t L^2})\norm[u-v]_{L^{4/3}([0,t], L^4)}
\\
&\cleq Mt^{1/2}\norm[u-v]_{L^4([0,t], L^4)}.
\end{align*}
From \eqref{eq:3.7} and these estimates, we obtain
\begin{align*}
\norm[u-v]_{\scX([0,t])}\cleq \frac{1}{\delta}\int_0^t\norm[(u-v)(\tau)]_{L^2}^{1-\delta}d\tau+Mt^{1/2}\norm[u-v]_{L^4([0,t], L^4)}.
\end{align*}
If we choose $T_0\in(0,T)$ small enough, the second term on the RHS can be absorbed into the left hand side (LHS), so that
\begin{align*}
\norm[(u-v)(t)]_{L^2}\leq \frac{C}{\delta}\int_0^t\norm[(u-v)(\tau)]_{L^2}^{1-\delta}d\tau, 
\quad t\in [0,T_0],
\end{align*}
where $C>0$ is a constant independent of $\delta\in(0,1)$.
Applying Lemma \ref{lem:3.4}, we obtain
\begin{align*}
\norm[(u-v)(t)]_{L^2}\leq \l( CT_0\r)^{1/\delta},\quad t\in[0,T_0].
\end{align*}
We rechoose $T_0$ so that $CT_0<1$, and take the limit $\delta\to0$ to obtain $u=v$ on $[0,T_0]$. We note that $T_0$ depends only on $M$, and therefore by repeating this argument a finite number of times, we conclude that $u=v$ on $[0,T]$.
\end{proof}

\begin{remark}
\label{rem:3.5}
We note that Lemma \ref{lem:3.4} was also used for the proof of uniqueness of cubic NLS on compact manifolds in \cite{BGT04ajm}, and the whole argument reducing to this Lemma was cited as Yudovitch's argument \cite{J63}. 
Although the source of the parameter $\delta$ is different, it would be of independent interest to note the similarity in the proof of uniqueness.
\end{remark}

\begin{remark}
\label{rem:3.6}
Bourgain \cite{B93} conjectured that 
\begin{align*}
2<p<2+\frac{4}{d} \implies\norm[U(t)u_0]_{L^p([0,T]\times\T^d)}\cleq
\norm[u_0]_{L^2(\T^d)}.
\end{align*}
If this estimate is proven for some ${p\in(2,2+4/d)}$, then our proof also applies to the case of $d\ge2$, and it follows that weak $L^2$ solutions of \eqref{eq:1.7} are unique.
\end{remark}

\subsection{Perturbation of pure power nonlinearities}
\label{sec:3.3}

In this subsection, we consider the equation \eqref{eq:1.10}.
Let $u\in L^\infty([0,T], H^1(\R^d))$ be a weak $H^1$ solution of \eqref{eq:1.10}. By using the cut-off function defined by \eqref{eq:2.1}, we deduce that
\begin{align*}
i\pt_t (\varphi_R u)+\Delta(\varphi_R u)-2\nabla u\cdot\nabla\varphi_R-u\Delta\varphi_R+\varphi_R\l( g(u)+h(u) \r)=0\quad\text{in}~H^{-1}(\R^d)
\end{align*}
for a.a. $t\in\R$, where we set
\begin{align*}
h(u)=\mu |u|^{\alpha}u. 
\end{align*}
It follows from Duhamel's formula that $u$ satisfies the integral equation
 \begin{align*}
\varphi_Ru(t)=
\begin{aligned}[t]
&U(t)\varphi_Ru_0-i\int_0^tU(t-\tau)\l(2\nabla\varphi_R\cdot\nabla u(\tau) +\Delta\varphi_R u(\tau)\r)d\tau
\\
&{}+i\int_0^t U(t-\tau)\varphi_Rg(u(\tau))d\tau+i\int_0^t U(t-\tau)\varphi_Rh(u(\tau))d\tau, \quad t\in [0,T].
\end{aligned}
\end{align*}
We are now in a position to prove Theorem \ref{thm:1.6}.
\begin{proof}[Proof of Theorem \ref{thm:1.6}]
For a given initial data $u_0\in H^1(\R^d)$, let $u,v$ be two weak $H^1$ solutions of \eqref{eq:1.10} on $[0,T]$.
We set 
\begin{align}
\label{eq:3.9}
M= \max\left\{ \norm[u]_{L^\infty_TH^1},  \norm[v]_{L^\infty_TH^1}\r\}.
\end{align}
It follows from the integral equation above that
\begin{align}
\label{eq:3.10}
\varphi_R(u-v)(t)=
\begin{aligned}[t]
e_R(t)&+i\int_0^tU(t-\tau)\varphi_R (g(u)-g(v) )(\tau) d\tau
\\
&{}~+i\int_0^tU(t-\tau)\varphi_R (h(u)-h(v) )(\tau) d\tau
\end{aligned}
\end{align}
for all $t\in[0,T]$, where we recall that $e_R(t)$ is defined by \eqref{eq:2.9}. We obtain from \eqref{eq:3.9} that
\begin{align}
\label{eq:3.11}
\norm[e_R]_{L^\infty_TL^2}\cleq \frac{M}{R}.
\end{align}

We will estimate the second and third terms on the RHS of \eqref{eq:3.10}.
Let $r=\alpha+2$ and let $q$ be such that $(q,r)$ is admissible. We introduce the function space
\begin{align*}
\scY(I)=L^\infty(I, L^2(\R^d))\cap L^q(I, L^r(\R^d)),
\end{align*}
where $I$ is a subinterval of $[0,T]$, and its norm 
\begin{align*}
\norm[u]_{\scY(I)}=\norm[u]_{L^\infty(I,L^2)}+\norm[u]_{L^q(I,L^r)}.
\end{align*}
Noting that 
\begin{align*}
\abs[h(u)-h(v)] \cleq (|u|^\alpha+|v|^{\alpha})\abs[u-v],
\end{align*}
we obtain from H\"older's inequality that
\begin{align*}
\norm[h(u)-h(v)]_{L^{q'}(I, L^{r'})} \cleq |I|^{\frac{1}{q'}-\frac{1}{q}}(\norm[u]_{L^\infty(I,L^r)}^\alpha+\norm[v]_{L^\infty(I,L^r)}^\alpha)\norm[u-v]_{L^q(I,L^r)}.
\end{align*}
Therefore, applying Lemma \ref{lem:3.3}, we obtain
\begin{align}
\label{eq:3.12}
\norm[ {\Phi[\varphi_R (h(u)-h(v))]} ]_{\scY([0,t])}
\cleq t^{\frac{1}{q'}-\frac{1}{q}}M^\alpha\norm[u-v]_{L^q([0,t],L^r)}.
\end{align}
From \eqref{eq:3.2} we obtain
\begin{align*}
\abs[g_2(u)-g_2(v)]&\cleq\l(|u|^\alpha+|v|^\alpha\r)|u-v|,
\end{align*}
and therefore, as in \eqref{eq:3.12}, we obtain
\begin{align}
\label{eq:3.13}
\norm[ {\Phi[\varphi_R (g_2(u)-g_2(v))]} ]_{\scY([0,t])}
\cleq t^{\frac{1}{q'}-\frac{1}{q}}M^\alpha\norm[u-v]_{L^q([0,t],L^r)}.
\end{align}
Regarding the estimate of ${g_1(u)-g_1(v)}$, we obtain from H\"older's inequality that
\begin{align*}
\norm[\varphi_R|u-v|^{1-\delta}]_{L^2}
\cleq\abs[B_R]^{\delta/2}\norm[\varphi_R(u-v)]_{L^2}^{1-\delta}.
\end{align*}
Combined with Lemma \ref{lem:3.3}, this yields that
\begin{align}
\label{eq:3.14}
\norm[ {\Phi[g_1(u)-g_1(v)]} ]_{\scY([0,t])}
\cleq \frac{\abs[B_R]^{\delta/2}}{\delta}\int_0^t\norm[\varphi_R(u-v)(\tau)]_{L^2}^{1-\delta}d\tau.
\end{align}

Now we deduce from \eqref{eq:3.11}--\eqref{eq:3.14} that
\begin{align*}
\norm[\varphi_R(u-v)]_{\scY([0,t])}\cleq_M \frac{1}{R} +\frac{|B_R|^{\delta/2}}{\delta}\int_0^t\norm[\varphi_R(u-v)(\tau)]_{L^2}^{1-\delta}d\tau+t^{\frac{1}{q'}-\frac{1}{q}}\norm[\varphi_R(u-v)]_{L^q([0,t],L^r)}
\end{align*}
for all $t\in[0,T]$. If we choose $T_0\in(0,T)$ small enough, the third term on the RHS can be absorbed into the LHS, so that
\begin{align}
\label{eq:3.15}
\norm[\varphi_R(u-v)(t)]_{L^2}\leq \frac{C_1}{R}+\frac{C_2|B_R|^{\delta/2} }{\delta}\int_0^t\norm[\varphi_R(u-v)(\tau)]_{L^2}^{1-\delta}d\tau
\end{align}
for all $t\in[0,T_0]$, where the positive constants $C_1,C_2$ are independent of parameters $\delta\in(0,1)$ and $R>0$. We now set $\delta=\frac{1}{\log R}$. Then we deduce that $|B_R|^{\delta/2}\cleq 1$ and
\begin{align*}
\lim_{R\to\infty}\l(\frac{C_1}{R}\r)^{\delta}=\lim_{R\to\infty}\l(\frac{C_1}{R}\r)^{1/\log R}=\frac{1}{e}.
\end{align*}
Applying Lemma \ref{lem:3.4} to \eqref{eq:3.15}, we obtain
\begin{align*}
\norm[\varphi_R(u-v)(t)]_{L^2}
\leq \l( \l(\frac{C_1}{R}\r)^{1/\log R} +C_3T_0 \r)^{\log R},\quad t\in[0,T_0].
\end{align*}
We rechoose $T_0$ so that $C_3T_0<1-e^{-1}$, and by taking the limit $R\to\infty$, the RHS converges to $0$. This yields that $u=v$ on $[0,T_0]$. Since $T_0$ depends only on $M$, we get the conclusion.
\end{proof}

\section*{Acknowledgments}

This work is supported by JSPS KAKENHI Grant Number JP24H00024.

\let\l\oldl


\end{document}